\documentclass[letterpaper,10pt,conference]{ieeeconf}
\IEEEoverridecommandlockouts % This command is only needed if you want to use the \thanks command
\usepackage{amsmath,amssymb,mathtools}
\usepackage[dvipsnames]{xcolor}
\usepackage{graphicx}
\usepackage{booktabs}
\usepackage{array}
\usepackage{stmaryrd}
\usepackage[compress]{cite}
\usepackage[multiple]{footmisc}

\newtheorem{theorem}{Theorem}

\newtheorem{lemma}[theorem]{Lemma}
\newtheorem{proposition}[theorem]{Proposition}

\newtheorem{example}{Example}

\newtheorem{assumption}{Assumption}

\newcommand{\proofname}{Proof}

\title{\LARGE\bf Finite Data-Rate Feedback Stabilization of Continuous-Time Switched Linear Systems with Unknown Switching Signal}
\author{Guillaume~O.~Berger and Rapha\"el~M.~Jungers%
\thanks{G.~Berger is a FNRS/FRIA Fellow.
R.~Jungers is a FNRS Research Associate.
He is supported by the Walloon Region and the Innoviris Foundation.
Both are with ICTEAM institute, UCLouvain, Louvain-la-Neuve, Belgium.
\texttt{\{guillaume.berger,raphael.jungers\}@uclouvain.be}.}}

\newcommand{\new}{\color{blue}}
\newcommand{\extended}[1]{{\color{orange}#1}}
\newcommand{\short}[1]{{\color{ForestGreen}#1}}
\renewcommand{\extended}[1]{#1}
\renewcommand{\short}[1]{#1}
\renewcommand{\short}[1]{}

\newcounter{changes}

\renewcommand{\Re}{\mathbb{R}}

\newcommand{\NNb}{\mathbb{N}}

\newcommand{\mynorm}{\lVert\cdot\rVert}
\newcommand{\nint}{\mathrm{int}}
\newcommand{\diff}{\mathrm{d}}

\newcommand{\calE}{\mathcal{E}}

\newcommand{\calK}{\mathcal{K}}

\newcommand{\calQ}{\mathcal{Q}}

\newcommand{\calU}{\mathcal{U}}

\newcommand{\dd}{{d\times d}}
\newcommand{\xh}{\hat{x}}

\newcommand{\mh}{\hat{m}}
\newcommand{\setMatrix}{\small\arraycolsep=0.5\arraycolsep}

\newcommand{\nmissed}{N^{\mathrm{sw}}}
\newcommand{\sigmah}{\hat{\sigma}}
\newcommand{\epsb}{\bar{\varepsilon}}
\newcommand{\xt}{\tilde{x}}
\newcommand{\alphab}{\bar{\alpha}}
\newcommand{\rhob}{\bar{\rho}}

\newcommand{\indfunc}{\mathbf{1}}

\begin{document}

\maketitle
\thispagestyle{empty}
\pagestyle{empty}

%%%%%%%%%%%%%%%%%%%%%%%%%%%%%%%%%%%%%%%%%%%%%%%%%%%%%%%%%%%%%%%%%%%%%%%%%%%%%%%%
\begin{abstract}
In this paper, we study the problem of stabilizing switched linear systems when only limited information about the state and the mode of the system is available, which occurs in many applications involving networked switched systems (such as cyber-physical systems, IoT, etc.).
First, we show that switched linear systems with arbitrary switching, i.e., with no constraint on the switching signal, are in general not stabilizable with a finite data rate.
Then, drawing on this result, we restrict our attention to systems satisfying a fairly mild slow-switching assumption, in the sense that the switching signal has an average dwell time bounded away from zero.
We show that under this assumption, switched linear systems that are stabilizable in the classical sense remain stabilizable with a finite data rate.
A practical coder--controller that stabilizes the system is presented and its applicability is demonstrated on numerical examples.
\end{abstract}

%%%%%%%%%%%%%%%%%%%%%%%%%%%%%%%%%%%%%%%%%%%%%%%%%%%%%%%%%%%%%%%%%%%%%%%%%%%%%%%%
\section{Introduction}\label{sec-intro}

This paper studies two important and challenging features of modern control systems: data-rate constraints and switching.
Many modern control systems (such as IoT, networked systems, etc.) involve spatially distributed components that communicate through a shared, digital communication network, that can carry only a finite amount of information per unit of time.
This limitation on the information flow can have large negative effects on the performance of the control loop.
This has motivated a considerable amount of research to study control problems subject to data-rate constraints, as surveyed in \cite{hespanha2007asurvey,matveev2009estimation,zhongping2013quantized}.
On the other hand, many systems encountered in practice involve switching between different operation modes; e.g., due to the interaction of physical processes and digital devices (as in cyber-physical systems), external influences (e.g., human in the loop), discontinuous dynamics (e.g., physical processes with impact), the nature of the controller (e.g., logic-dynamic controllers), etc.
Control problems involving switching have also attracted a lot of attention from the control community in recent years; see, e.g., the surveys \cite{van2000anintroduction,liberzon2003switching,shorten2007stability}, and the references therein.

This paper focuses on the problem of limited data-rate stabilization of continuous-time Switched Linear Systems (SLSs).
These are systems described by a finite set of linear modes, among which the system can switch in time (see Figure~\ref{fig-coder-controller-switched} for a representation).
As paradigmatic examples of hybrid and cyber-physical systems, SLSs naturally appear in many engineering applications, or as abstractions of more complex dynamical systems \cite{lin2009stability,jungers2009thejoint,sun2011stability}.

\begin{figure}
\centering
\includegraphics[width=0.95\linewidth]{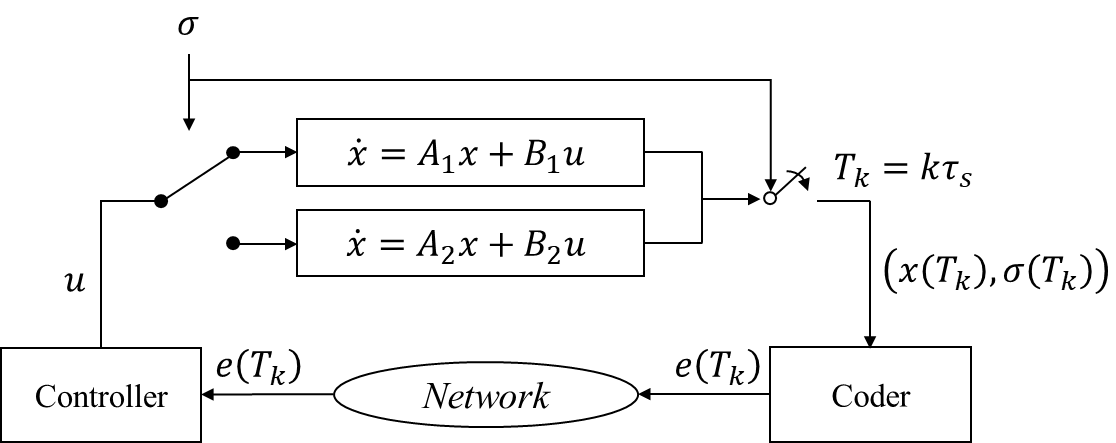}
\caption{Control of switched linear systems over limited, digital communication networks.}
\label{fig-coder-controller-switched}
\vspace{-10pt}
\end{figure}

Although control of switched systems and control with data-rate constraints have been two active areas of research for some time now, the study of control problems involving switching and data-rate constraints simultaneously seems to have not received much attention so far.
(Some work has been devoted to the stabilization of Markov jump linear systems with data-rate constraints \cite{nair2003infimum,zhang2009stabilization,ling2010necessary,xiao2010stabilization}.
However, the information structure considered in these references implies that the mode of the system is always known to the controller, so that the problems of switching and state estimation with limited information can be treated separately.)
Combining these two aspects in a unified framework is however essential if we want to address control problems encountered in a wide range of applications involving networked switched systems, which generally imply that the controller has limited information on both the state \emph{and} the mode of the system.

When both state observation and switching signal observation are subject to data-rate constraints, the problems of state estimation and switching are intrinsically coupled.
For instance, state encoding strategies must take into account the fact that unobserved switching may occur during the sampling interval.
In particular, we will see that switched systems with unconstrained switching signal have in general an \emph{infinite stabilization entropy}, meaning that they are not stabilizable with any finite data rate.
On the other hand, it is a standard technique in stability and stabilizability analysis of switched systems to impose \emph{slow-switching conditions}---generally described by a dwell time and/or an average dwell time (ADT) \cite{hespanha1999stability,liberzon2003switching}---on the switching signal to reduce its expressiveness.
Recently, these techniques were used in the context of limited data-rate control of SLSs \cite{liberzon2014finite,yang2017feedback,wakaiki2014output}.

Our work draws upon these references, especially \cite{liberzon2014finite}, for the formulation of the problem of interest, namely the stabilization of SLSs under data-rate constraints and subject to slow-switching assumptions.
However, we consider different objectives regarding the design of a control strategy with finite data rate: while in \cite{liberzon2014finite,yang2017feedback,wakaiki2014output}, the ADT is used as a design parameter to ensure stabilization of the system with data-rate constraints, our goal here is to study the limited data-rate stabilization of SLSs with \emph{arbitrary} ADT.

The contribution of this paper is twofold.
First, we show that SLSs with \emph{no constraint} on the switching signal, are in general not stabilizable with a finite data rate.
We present an example of a SLS that is feedback stabilizable for any switching signal in the absence of data-rate constraints, but cannot be stabilized with a finite data rate.
This motivates the introduction of slow-switching assumptions in order to make the problem of limited data-rate stabilization of SLSs tractable, as otherwise the coder cannot transmit information fast enough to the controller to achieve stabilization.

Secondly, we show that under a fairly mild slow-switching assumption on the switching signal, SLSs that are stabilizable in the classical sense remain stabilizable with a finite data rate.
More precisely, we show that any stabilizable (without data-rate constraints) SLS with ADT bounded away from zero can be stabilized by a coder--controller with finite data rate.
We present a sufficient upper bound on the data rate depending on the system and the ADT, and we describe the implementation of a coder--controller that stabilizes the system.
We stress out that in our analysis (unlike \cite{liberzon2014finite,yang2017feedback,wakaiki2014output}) the ADT is fixed, so that the controller has no influence on the value of the ADT.

The paper is organized as follows.
The problem of interest, including the definition of SLSs, the basic assumptions on the system and the concept of coder--controller, is formulated in Section~\ref{sec-problem}.
Our main results are stated in Section~\ref{sec-main-results}.
Then, in Section~\ref{sec-coder-controller}, we describe the implementation of a coder--controller that stabilizes the system.
Finally, in Section~\ref{sec-numerical}, we illustrate the usage of the coder--controller with a numerical example.

\emph{Notation}.
For vectors, $\mynorm$ denotes the Euclidean $2$-norm, and for matrices it denotes the associated matrix norm (i.e., $\lVert M\rVert=$ largest singular value of $M$).
$B(\xi,r)$ is the closed ball centered at $\xi\in\Re^d$ with radius $r\geq0$.
% $\lceil\cdot\rceil$ denotes the \emph{ceil} operator.
% If $A$ and $B$ are sets, then $B^A$ denotes the set of functions from $A$ to $B$.
If $f:A\to B$, and $A'\subseteq A$, then $f\vert_{A'}$ denotes the restriction of $f$ to the domain $A'$.
A function $g:\Re_{\geq0}\to\Re_{\geq0}$ is of class-$\calK$ if it is strictly increasing, continuous and $g(0)=0$.

%%%%%%%%%%%%%%%%%%%%%%%%%%%%%%%%%%%%%%%%%%%%%%%%%%%%%%%%%%%%%%%%%%%%%%%%%%%%%%%%
\section{Problem formulation}\label{sec-problem}

%%%%%%%%%%%%%%%%%%%%%%%%%%%%%%%%%%%%%%%%%%%%%%%%%%%%%%%%%%%%%%%%%%%%%%%%%%%%%%%%
\subsection{Switched linear systems}\label{ssec-SLSs}

Consider a continuous-time \emph{Switched Linear System} (SLS) with affine control input:
\begin{equation}\label{eq-SLS-cont}
\dot{x}(t) = A_{\sigma(t)}x(t) + B_{\sigma(t)}u(t), \quad x(0)\in K, \quad t\geq0,
\end{equation}
where $\sigma(t)\in\Sigma\coloneqq\{1,\ldots,N\}$ and $u(t)\in\Re^c$, $A_i\in\Re^\dd$ and $B_i\in\Re^{d\times c}$ for all $i\in\Sigma$, and $K\subseteq\Re^d$ is a compact set with $0\in\nint(K)$.
The function $\sigma:\Re_{\geq0}\to\Sigma$ is called the \emph{switching signal} (or \emph{s.s.} for short) and is assumed to be piecewise constant and right-continuous.

The discontinuity points of $\sigma$ are  called the ``switching times'' or simply ``switches''.
For $t\geq s\geq 0$, we let $N_\sigma(t,s)$ be the number of switches of $\sigma$ in the interval $[s,t)$.

As we will see in Subsection~\ref{ssec-arbitrary-switching}, SLSs under arbitrary switching are in general not tractable for the problem of stabilization with limited data rate.
Therefore, in our analysis, we make the assumption that the system is not switching too fast, in the following sense:

\begin{assumption}\label{ass-ADT}
There is $\tau_a>0$, called the \emph{Average Dwell Time} (ADT), and a constant $N_0\geq0$ such that
\[
N_\sigma(t,s)\leq N_0 + \frac{t-s}{\tau_a} \qquad \forall\,t\geq s\geq 0.
\]
The parameter $N_0$ is fixed but not known by the controller a priori.
\end{assumption}

The concept of ADT, introduced in \cite{hespanha1999stability}, has become a standard assumption in the study of stability and stabilizability of switched and hybrid systems \cite{liberzon2003switching,sun2011stability}.
It has also received attention in the context of control of switched systems with limited information \cite{zhang2009stabilization,liberzon2014finite,yang2017feedback,wakaiki2014output}.

Our goal is to stabilize \eqref{eq-SLS-cont} under data-rate constraints.
Clearly, a necessary condition is that the system is stabilizable in the absence of data-rate constraints.
Hence, we make the following assumption on the stabilizability of \eqref{eq-SLS-cont}:

\begin{assumption}\label{ass-feedback-stabilizable}
There is a \emph{feedback law} $\varphi:\Re^d\times\Sigma\to\Re^c$, positively homogeneous in the $1$st argument and piecewise continuous, and constants $D\geq0$ and $\mu_1,\mu_2>0$ such that (i) $\mu_1/\tau_a<\mu_2$, and (ii) for every s.s.\ $\sigma$, the solution of the closed-loop system $\dot{x}(t)=A_{\sigma(t)}x(t)+B_{\sigma(t)}\varphi(x(t),\sigma(t))$ satisfies
\begin{equation}\label{eq-feedback-stabilizable-convergence}
\lVert x(t)\rVert \leq D\lVert x(0)\rVert\,e^{\mu_1N_\sigma(t,0)-\mu_2t}\qquad \forall\,t\geq 0.
\end{equation}
We assume that the feedback law $\varphi(\cdot,\cdot)$ and the parameters $D,\mu_1,\mu_2$ are fixed and known by the controller.
\end{assumption}

Assumption~\ref{ass-feedback-stabilizable} implies that the closed-loop system \eqref{eq-SLS-cont} with feedback control input $u(t)=\varphi(x(t),\sigma(t))$ is asymptotically stable, for every s.s.\ satisfying Assumption~\ref{ass-ADT}.
The existence of a feedback control law satisfying \eqref{eq-feedback-stabilizable-convergence} can be ensured for instance if the system admits a \emph{multiple Control Lyapunov Function} (CLF) \cite{liberzon2003switching,lin2009stability}.
% This approach was used in \cite{liberzon2014finite,yang2017feedback,wakaiki2014output} to derive a sufficient lower bound on the ADT to stabilize SLSs under data-rate constraints.
An interesting situation, on which we will come back in  Subsection~\ref{ssec-comparison}, is when the system admits a \emph{common} CLF; in this case, \eqref{eq-feedback-stabilizable-convergence} is satisfiable with $\mu_1=0$, so that Assumption~\ref{ass-feedback-stabilizable} holds for any $\tau_a>0$.

% Condition (ii) expresses that when there is no switch, the norm of the state decreases \emph{asymptotically} exponentially with rate $-\mu_2$.
% However, after a switch, the norm of the state can increase by a factor up to $e^{\mu_1}$.\footnote{This does not mean that there is a jump in the state but simply reflects the fact that since the dynamic of the system changed, the norm of the state can increase momentarily before returning to an exponential decrease.}
% Finally, condition (i), together with Assumption~\ref{ass-ADT} on $N_\sigma(\cdot,0)$, implies that the system with feedback control law $\varphi(\cdot,\cdot)$ is asymptotically stable.

%%%%%%%%%%%%%%%%%%%%%%%%%%%%%%%%%%%%%%%%%%%%%%%%%%%%%%%%%%%%%%%%%%%%%%%%%%%%%%%%
\subsection{Coder--controller}\label{ssec-coder-controller}

We investigate the problem of stabilizing system \eqref{eq-SLS-cont} when direct observation of the system is not possible.
Information about the mode and the state of the system will thus be delivered by a \emph{coder} connected to a \emph{controller} via a digital channel that can carry only a finite amount of information per unit of time.
The situation is depicted in Figure~\ref{fig-coder-controller-switched}.
At periodic sampling times $T_k\coloneqq k\tau_s$, $k=0,1,2,3,\dots$, the coder observes the current state and mode of the system, and sends one discrete-valued symbol $e(T_k)$, selected from a finite coding alphabet $\calE_k$, to the controller.
Neglecting transmission errors and delay, at time $T_k$ the controller has the symbols $e(0),\ldots,e(k)$ available and it generates a control input $u(\cdot)$ for the coming epoch $[T_k,T_{k+1})$.

More precisely, let $\tau_s>0$ be a sampling period and $(\calE_k)_{k\in\NNb}$ a sequence of coding alphabets.
The symbol sent by the coder at time $T_k\coloneqq k\tau_s$ is defined by
\begin{equation}\label{eq-symbol-coder-controller}
e(T_k) = \gamma_k(x(T_0),\ldots,x(T_k);\sigma(T_0),\ldots,\sigma(T_k)),
\end{equation}
where $\gamma_k: (\Re^d)^k \times \Sigma^k \to \calE_k$ is the coder function at time $T_k$ and $x(\cdot)$ is the state of the system.
Assuming the channel is noiseless and without delay, at time $T_k$ the controller has the symbols $e(T_0),\ldots,e(T_k)$ available and it generates the input function $u(\cdot)$ defined on the interval $[T_k,T_{k+1})$ by:
\begin{equation}\label{eq-input-coder-controller}
u(\cdot)\vert_{[T_k,T_{k+1})} = \zeta_k(e(T_0),\ldots,e(T_k)),
\end{equation}
where $\zeta_k : \calE_0\times\ldots\times\calE_k\to (\Re^c)^{[T_k,T_{k+1})}$ is the controller function at time $T_k$, and it is assumed that $u(\cdot)\vert_{[T_k,T_{k+1})}$ is integrable (i.e., $L^1$).
Let $\gamma=(\gamma_k)_{k\in\NNb}$ and $\zeta=(\zeta_k)_{k\in\NNb}$.
The pair $(\gamma,\zeta)$ is called a \emph{coder--controller}.
% (note that $\tau_s$ and $(\calE_k)_{k\in\NNb}$ are defined implicitly in $\gamma$).

At each time $T_k$, the symbol $e(T_k)$ is transmitted via a limited communication channel.
Using binary representation of the symbols, the \emph{averaged communication data rate} (or \emph{data rate} for short) [in bits per unit of time] of the coder--controller is given by
\[
% \textstyle
R(\gamma,\zeta) = \limsup_{k\to\infty}\,\frac{1}{k\tau_s}\,{\sum}_{j=0}^{k-1}\,\log_2\,\lvert\calE_j\rvert.
\]

In modern applications involving networked systems, the communication capacity of the network is generally important.
However, there are usually many resources competing for the same bandwidth (as for instance in IoT applications).
Therefore, it is essential to have coding--controlling strategies with an averaged communication data rate as small as possible.
The question of determining the smallest data rate that is needed to achieve a given control task is also intriguing from the theoretical point of view.

%%%%%%%%%%%%%%%%%%%%%%%%%%%%%%%%%%%%%%%%%%%%%%%%%%%%%%%%%%%%%%%%%%%%%%%%%%%%%%%%
\section{Main Results}\label{sec-main-results}

%%%%%%%%%%%%%%%%%%%%%%%%%%%%%%%%%%%%%%%%%%%%%%%%%%%%%%%%%%%%%%%%%%%%%%%%%%%%%%%%
\subsection{Finite data-rate stabilization}

The control objective studied in this paper is the stabilization of system \eqref{eq-SLS-cont}, defined in Subsection~\ref{ssec-SLSs}, under data-rate constraints.
More precisely, we want to show the existence of a coder--controller, as described in Subsection~\ref{ssec-coder-controller}, with finite data rate, that stabilizes the system:

\begin{theorem}\label{thm-coder-controller-stabilizes}
Consider system \eqref{eq-SLS-cont} and let Assumptions~\ref{ass-ADT}--\ref{ass-feedback-stabilizable} hold.
There is a coder--controller $(\gamma,\zeta)$ with $R(\gamma,\zeta)<\infty$ that satisfies the following properties:\footnote{A coder--controller that satisfies these properties will be said to \emph{stabilize} the system.} there is $\lambda>0$ and a class-$\calK$ function $g(\cdot)$ such that every trajectory $x(\cdot)$ of \eqref{eq-SLS-cont} with control input $u(\cdot)$ defined by \eqref{eq-symbol-coder-controller}--\eqref{eq-input-coder-controller} satisfies
\[
\lVert x(t)\rVert\leq g(\lVert x(0)\rVert)\,e^{-\lambda t}\qquad \forall\,t\geq0.
\]\vskip0pt
\end{theorem}

We will provide a constructive proof of Theorem~\ref{thm-coder-controller-stabilizes}.
More precisely, in Section~\ref{sec-coder-controller}, we describe the implementation of a finite data-rate coder--controller achieving stabilization.
A precise upper bound on the data rate will be derived in due course of the description of the coder--controller; see \eqref{eq-data-rate-coder-controller} in Section~\ref{sec-coder-controller}.
As for the decay rate $\lambda$, it will be obtained in the proof that the proposed coder--controller stabilizes the system; \extended{see \eqref{eq-lambda-def} in Section~\ref{sec-coder-controller}}\short{see Eq.~\eqref{eq-lambda-def} in the extended version of this work \cite{berger2020finite}}.
The function $g(\cdot)$, however, will not be explicitly characterized but its existence will be demonstrated.
As a class-$\calK$ function, $g(\cdot)$ satisfies $g(r)\to0$ when $r\to0$.
However, as it will be clear from the proof of its existence, $g(r)$ is not Lipschitz continuous at $r=0$.
This lack of regularity is not due to potential sub-optimality of the proposed coder--controller, or to the switching nature of the system, but is intrinsic to \emph{any} finite-data-rate stabilization scheme for linear systems (including LTI systems); see for instance \cite[Proposition~2.2]{colonius2012minimal}.

%%%%%%%%%%%%%%%%%%%%%%%%%%%%%%%%%%%%%%%%%%%%%%%%%%%%%%%%%%%%%%%%%%%%%%%%%%%%%%%%
\subsection{Necessity of the average dwell time}\label{ssec-arbitrary-switching}

In this subsection, we would like to stress out the importance of Assumption~\ref{ass-ADT} in Theorem~\ref{thm-coder-controller-stabilizes}.
Therefore, we show with a simple example that SLSs under \emph{arbitrary} switching (i.e., with ADT equal to zero) are in general not stabilizable with a finite data rate:

\begin{example}\label{exa-arbitrary-not-tractable}
Consider system \eqref{eq-SLS-cont} with $d=1$, $\Sigma=\{1,2\}$, and matrices $A_1 = A_2 = 0$, $B_1 = -1$ and $B_2 = 1$.
This system is somehow the most basic affine-controlled SLS, and it clearly satisfies Assumption~\ref{ass-feedback-stabilizable} with $\mu_1=0$: for instance, take $\varphi(x(t),\sigma(t))=-B_{\sigma(t)}x(t)$.
\end{example}

\begin{proposition}\label{pro-arbitrary-not-tractable}
System \eqref{eq-SLS-cont}, with $A_\Sigma$ and $B_\Sigma$ as in Example~\ref{exa-arbitrary-not-tractable} and under arbitrary switching, is not stabilizable by a coder--controller with finite data rate.
\end{proposition}

\begin{proof}
Assume the contrary and let $(\gamma,\zeta)$ be a coder--controller that stabilizes the system.
For some $T>0$ fixed, let $\calU_T$ be the set of all distinct input functions $u(\cdot)$ that can be generated by the coder--controller, i.e., by \eqref{eq-symbol-coder-controller}--\eqref{eq-input-coder-controller}, on the interval $[0,T)$.
Since $R(\gamma,\zeta)<\infty$, it follows that $\calU_T$ is finite, and by \eqref{eq-input-coder-controller}, $\calU_T$ contains only $L^1$ functions.

Now, for each $n\in\NNb_{>0}$, let $\sigma_n(\cdot)$ be the s.s.\ that oscillates between mode $1$ and mode $2$ with frequency $2/n$: that is, $\sigma_n(t)=1$ if $t\in[0,1/n)+2\NNb/n$, and $\sigma_n(t)=2$ if $t\in[1/n,2/n)+2\NNb/n$.
Then, by using an adaptation of the proof of the Riemann--Lebesgue lemma \cite[Corollary~14.5]{teschltopics}, one can show that for any $L^1$ function $f:[0,T)\to\Re$, it holds that $\int_0^T B_{\sigma_n(t)}f(t)\,\diff t\to 0$ when $n\to\infty$.
Since $\calU_T$ is finite, this implies that for any $\varepsilon>0$, there is $n\in\NNb_{>0}$ such that $\lvert\int_0^T B_{\sigma_n(t)}u(t)\,\diff t\rvert<\varepsilon$ for all $u\in\calU_T$.

Thus, for every $u\in\calU_T$, $\lvert x(T)\rvert>\lvert x(0)\rvert-\varepsilon$ where $x(\cdot)$ is a trajectory of \eqref{eq-SLS-cont} with s.s.\ $\sigma_n$ and with input $u$.
Since $T$ and $\varepsilon$ are arbitrary and $K$ contains at least one point $x(0)\neq0$ (since it has nonempty interior), this is a contradiction with the hypothesis that $(\gamma,\zeta)$ stabilizes the system.
\end{proof}

Summarizing, the above example shows that for the problem of limited data-rate stabilization of SLSs to be tractable, the switching signal cannot switch too rapidly (or at least not too rapidly during a too long period), as otherwise the uncertainty on the mode of the system, and especially on $B_\sigma$, between two transmission times ($T_k$, $T_{k+1}$) will be so large that the the system cannot be stabilized with a finite set of inputs.
This motivates the introduction of Assumption~1 which reduces the set of admissible switching signals by restricting the number of switches in bounded intervals.

%%%%%%%%%%%%%%%%%%%%%%%%%%%%%%%%%%%%%%%%%%%%%%%%%%%%%%%%%%%%%%%%%%%%%%%%%%%%%%%%
\subsection{Comparison with other works}\label{ssec-comparison}

Our work is strongly connected with \cite{liberzon2014finite,yang2017feedback,wakaiki2014output}, where the problem of limited data-rate stabilization of SLSs with constraints on the ADT is also considered.
The objectives regarding the design of a control strategy with finite data rate are however different.
In \cite{liberzon2014finite,yang2017feedback,wakaiki2014output}, the sampling frequency and the data rate of the coder--controller are fixed, and the objective is to obtain a sufficient lower bound on the ADT of the switching signal to ensure stabilization of the system.
In our work, however, we seek to obtain a coder--controller, with suitable sampling frequency and data rate, that stabilizes the system for any ADT such that Assumption~\ref{ass-feedback-stabilizable} holds.
In particular, if Assumption~\ref{ass-feedback-stabilizable} holds for every $\tau_a>0$, then for any given value of the ADT, we describe a coder--controller that stabilizes the system.
By contrast, in \cite{liberzon2014finite,yang2017feedback,wakaiki2014output}, even if the system admits a common CLF (which implies that Assumption~\ref{ass-feedback-stabilizable} holds for every $\tau_a>0$; see Subsection~\ref{ssec-SLSs}), the lower bound on the ADT does not converge to zero, even if the sampling frequency and the data rate tend to infinity.
The conservativeness of the lower bound on the ADT proposed in \cite{liberzon2014finite,yang2017feedback,wakaiki2014output} is mainly due to the fact that its derivation is based on the decrease of a multiple Lyapunov-like function for the sampled system.
In particular, the definition of this multiple Lyapunov-like function implies that each change of mode causes a nonzero ``jump'' in the value of the function; this results in nonzero lower bounds on the ADT, even if the system admits a common quadratic CLF; see, e.g., \cite[Eq.~(38)]{liberzon2014finite}.
By contrast, our approach relies on the convergence of a reference trajectory (by using Assumption~2 and a suitable sampling frequency) and the guarantee that the true state of the system does not stray too far from the reference trajectory (thanks to a suitable data rate); see Section~\ref{sec-coder-controller} below.
Last but not least, another difference of our framework with \cite{liberzon2014finite,yang2017feedback,wakaiki2014output} is that we do not impose any condition on the \emph{absolute} dwell time of the system.

%%%%%%%%%%%%%%%%%%%%%%%%%%%%%%%%%%%%%%%%%%%%%%%%%%%%%%%%%%%%%%%%%%%%%%%%%%%%%%%%
\section{Description of the coder--controller}\label{sec-coder-controller}

In this section, we describe the implementation of a coder--controller that stabilizes system~\eqref{eq-SLS-cont} under Assumptions~\ref{ass-ADT}--\ref{ass-feedback-stabilizable}.
The section is organized as follows.
First, we discuss the selection of the parameters of the coder--controller, which depend on the system and the quantities appearing in Assumptions~\ref{ass-ADT}--\ref{ass-feedback-stabilizable}.
Then, we present the implementations of the coder and the controller.
Finally, we show that the proposed coder--controller satisfies the assertions of Theorem~\ref{thm-coder-controller-stabilizes}.

%%%%%%%%%%%%%%%%%%%%%%%%%%%%%%%%%%%%%%%%%%%%%%%%%%%%%%%%%%%%%%%%%%%%%%%%%%%%%%%%
\subsection{Parameters definition}\label{ssec-parameters-definition}

Let $\nu = \frac{1}{2}\max_{i\in\Sigma}\,\lambda_{\mathrm{max}}(A_i^{}+A_i^\top)$, and let
\[
\Delta_1 = \max_{i,j\in\Sigma}\,\lVert A_i-A_j\rVert,\quad\Delta_2=\max_{i,j\in\Sigma}\,\lVert B_i-B_j\rVert.
\]
Also, define
\[
L=\max\,\{\,\lVert\varphi(\xi,i)\rVert : i\in\Sigma,\:\xi\in\Re^d,\:\lVert\xi\rVert=1\}.
\]
Pick $\tau_s>0$, $\alpha>0$ and $n\in\NNb$ such that
\begin{equation}\label{eq-parameters}
De^{-\mu_2 n\tau_s} + e^{\nu n\tau_s}\alpha + \varepsilon(n,\tau_s) < e^{-\mu_1n\tau_s/\tau_a}
\end{equation}
where $\varepsilon(n,\tau_s) = e^{\nu n\tau_s}\tau_s\frac{n\tau_s}{\tau_a}D(\Delta_1+\Delta_2L)$.%
\footnote{A strategy for choosing $\tau_s,\alpha,n$ could be: first, choose $T_s=n\tau_s$ large enough so that $De^{(\mu_1/\tau_a-\mu_2)T_s}<1$.
Then, for this $T_s$, choose $\alpha,n$ such that $e^{\nu T_s}\alpha$ and $\varepsilon(n,T_s/n)$ are small enough for \eqref{eq-parameters} to be satisfied.}

We will need the following lemma:

\begin{lemma}\label{lem-m-point-quantizer}
Let $\alpha>0$.
There is an $m$-point quantizer $Q:\Re^d\to\calQ\subseteq B(0,1)$ satisfying (i) $\lVert\xi-Q(\xi)\rVert\leq\alpha$ if $\lVert\xi\rVert\leq1$, (ii) $Q(\xi)=0$ if $\lVert\xi\rVert\leq\alpha/d^{1/2}$, and (iii)
\begin{equation}\label{eq-m-point-quantizer}
m=\lvert\calQ\rvert\leq\mh_\alpha\coloneqq\left(2\left\llbracket\frac{d^{1/2}}{2\alpha}\right\rrbracket + 1\right)^d.
\end{equation}
where $\llbracket\cdot\rrbracket$ is the \emph{rounding} (to the nearest integer) operator.
\end{lemma}

\extended{

\begin{proof}
See Appendix~\ref{ssec-lem-m-point-quantizer-proof}.
\end{proof}

}
\short{

\begin{proof}
See the extended version of this paper \cite{berger2020finite}.
\end{proof}

}

We will show that there is a coder--controller $(\gamma,\zeta)$, with sampling period $\tau_s$, that stabilizes the system and operates at data rate
\begin{equation}\label{eq-data-rate-coder-controller}
\kern-2pt R(\gamma,\zeta) = \frac{1}{\tau_s}\left[\frac{1}{n}\log_2\mh_\alpha + \frac{1}{n}\log_2(n+1) + \log_2\,\lvert\Sigma\rvert\right]\kern-2pt
\end{equation}
where $\mh_\alpha$ is as in \eqref{eq-m-point-quantizer}.

%%%%%%%%%%%%%%%%%%%%%%%%%%%%%%%%%%%%%%%%%%%%%%%%%%%%%%%%%%%%%%%%%%%%%%%%%%%%%%%%
\subsection{Coder implementation}\label{ssec-coder-implementation}

Let $\tau_s$, $\alpha$ and $n$ be as above, and let $Q(\cdot)$ be the quantizer associated to $\alpha$ as in Lemma~\ref{lem-m-point-quantizer}.
Also fix $r_0>0$ such that $K\subseteq B(0,r_0)$.%
\footnote{The assumption that the initial state of the system lays in a compact set known from the coder--controller is made for convenience and simplicity of the description of the coder--controller.
This assumption can be removed by using a ``zooming-out'' procedure as in \cite{liberzon2014finite}.}
The implementation of the coder is described in Figure~\ref{fig-coder-implementation}.

The implementation deserves the following explanations.
At every time $kn\tau_s$, $k\in\NNb_{>0}$, the coder computes the value of $\nmissed_k$ which is defined as the smallest integer in $\{0,\ldots,n\}$ such that $\beta_k$, defined by
\begin{equation}\label{eq-beta-defintion}
\beta_k = e^{\mu_1\nmissed_k}\psi+\alphab+e^{\mu_1\nmissed_k}\nmissed_k\epsb
\end{equation}
where $\psi=De^{-\mu_2n\tau_s}$, $\alphab=e^{\nu n\tau_s}\alpha$, $\epsb=e^{\nu n\tau_s}\tau_sD(\Delta_1+\Delta_2L)$, satisfies $\lVert x(kn\tau_s)\rVert\leq\beta_kr_{k-1}$.
We will see in Subsection~\ref{ssec-proof-coder-controller-correctness} that such an $\nmissed_k$ always exists.
Using this $\beta_k$, the coder updates the value of $r_k$ according to $r_k=\beta_kr_{k-1}$.
If $k=0$, simply use $\nmissed_0=0$ and $r_0$.

Using the above quantities, at time $kn\tau_s$, $k\in\NNb$, the coder sends a symbol that encodes the following information: (i)\ an approximation $\eta_k$ of the current state $x(kn\tau_s)$ scaled by $1/r_k$, using the quantizer $Q(\cdot)$, (ii)\ the current mode of the system, $\sigma(kn\tau_s)$, and (iii)\ the value of $\nmissed_k$.
Since $\eta_k$ can take at most $\mh_\alpha$ different values, $\sigma(kn\tau_s)$ at most $\lvert\Sigma\rvert$ values, and $\nmissed_k$ at most $n+1$ different values, it holds that a coding alphabet $\calE_{kn}$ of size $\log_2\mh_\alpha+\log_2(n+1)+\log_2\lvert\Sigma\rvert$ is sufficient to encode the symbol $e(kn\tau_s)$.
After this, at times $(kn+j)\tau_s$, $j\in\{1,\ldots,n-1\}$, the coder observes the current mode of the system and sends a symbol that encodes this mode.
For this, a coding alphabet $\calE_{kn+j}$ of size $\log_2\lvert\Sigma\rvert$ is sufficient.
Hence, it follows that the averaged communication data rate of the coder is equal to \eqref{eq-data-rate-coder-controller}.

%%%%%%%%%%%%%%%%%%%%%%%%%%%%%%%%%%%%%%%%%%%%%%%%%%%%%%%%%%%%%%%%%%%%%%%%%%%%%%%%
\subsection{Controller implementation}\label{ssec-controller-implementation}

Let $\tau_s$, $\alpha$ and $n$ be as in Subsection~\ref{ssec-parameters-definition}, and $Q(\cdot)$ and $r_0>0$ be as in Subsection~\ref{ssec-coder-implementation}.
The implementation of the controller is described in Figure~\ref{fig-controller-implementation}.
See also Figure~\ref{fig-iterates}, where the different quantities appearing in the implementation of the controller are represented.

\begin{figure}
\hrule
\vskip3pt
\noindent{\bfseries Initialization:} Let $\nmissed_0=0$, and let $r_0$ be as in Subsection~\ref{ssec-coder-implementation}.

\noindent{\bfseries Loop:} at time $kn\tau_s$ {\bfseries for} $k=0,1,2,\ldots$
\begin{itemize}[\nocalcleftmargintrue\leftmargin=12pt]
    \item Observe $x(kn\tau_s)$ and $\sigma(kn\tau_s)$.
    \item {\bfseries If} $k>0$: Let $\nmissed_k$ be the smallest integer in $\{0,\ldots,n\}$ such that $\beta_k$ defined by \eqref{eq-beta-defintion} with this $\nmissed_k$ satisfies $\lVert x(kn\tau_s)\rVert\leq\beta_kr_{k-1}$.
    Let $r_k=\beta_kr_{k-1}$.
    \item Let $\eta_k = Q(x(kn\tau_s)/r_k)$.
    \item Send a symbol $e(kn\tau_s)$ to the controller that encodes the triple $(\eta_k,\sigma(kn\tau_s),\nmissed_k)$.
    \item {\bfseries Loop:} at time $(kn+j)\tau_s$ {\bfseries for} $j=1,2,\ldots,n-1$
    \begin{itemize}[\nocalcleftmargintrue\leftmargin=12pt]
        \item Observe $\sigma((kn+j)\tau_s)$.
        \item Send a symbol $e((kn+j)\tau_s)$ to the controller that encodes $\sigma((kn+j)\tau_s)$.
    \end{itemize}
\end{itemize}\vskip2pt
% \vskip1pt
\hrule
\caption{Coder implementation.}
\label{fig-coder-implementation}
\end{figure}
\begin{figure}
\hrule
\vskip3pt
\noindent{\bfseries Initialization:} Let $\nmissed_0=0$, and let $r_0$ be as in Subsection~\ref{ssec-coder-implementation}.

\noindent{\bfseries Loop:} at time $kn\tau_s$ {\bfseries for} $k=0,1,2,\ldots$
\begin{itemize}[\nocalcleftmargintrue\leftmargin=12pt]
    \item Receive symbol $e(kn\tau_s)$ and decode $(\eta_k,\sigma(kn\tau_s),\nmissed_k)$.
    \item {\bfseries If} $k>0$: let $\beta_k$ be defined as in \eqref{eq-beta-defintion} with $\nmissed_k$ obtained from the symbol, and let $r_k=\beta_kr_{k-1}$.
    \item Let $\xi_{kn}=r_k\eta_k$.
    \item {\bfseries For} $t\in[kn\tau_s,(kn+1)\tau_s)$: apply the input
    \[
    u(t) = \varphi(\xh_{kn}(t),\sigma(kn\tau_s))
    \]
    where $\xh_{kn}(\cdot)$ is the solution of the auxiliary system \eqref{eq-auxiliary-system} with the boundary condition $\xh_{kn}(kn\tau_s)=\xi_{kn}$ and with $i=\sigma(kn\tau_s)$.
    \item {\bfseries Loop:} at time $(kn+j)\tau_s$ {\bfseries for} $j=1,2,\ldots,n-1$.
    \begin{itemize}[\nocalcleftmargintrue\leftmargin=12pt]
        \item Receive symbol $e((kn+j)\tau_s)$ and decode $\sigma((kn+j)\tau_s)$.
        \item Let $\xi_{kn+j} = \xh_{kn+j-1}((kn+j)\tau_s)$.
        \item {\bfseries For} $t\in[(kn+j)\tau_s,(kn+j+1)\tau_s)$: apply the input
        \[
        u(t) = \varphi(\xh_{kn+j}(t),\sigma((kn+j)\tau_s))
        \]
        where $\xh_{kn+j}(\cdot)$ is the solution of the auxiliary system \eqref{eq-auxiliary-system} with the boundary condition $\xh_{kn+j}(kn\tau_s)=\xi_{kn+j}$ and with $i=\sigma((kn+j)\tau_s)$.
    \end{itemize}\vskip0pt
\end{itemize}
% \vskip1pt
\hrule
\caption{Controller implementation.}
\label{fig-controller-implementation}
\vspace{-10pt}
\end{figure}

\begin{figure}
\centering
\includegraphics[width=0.9\linewidth]{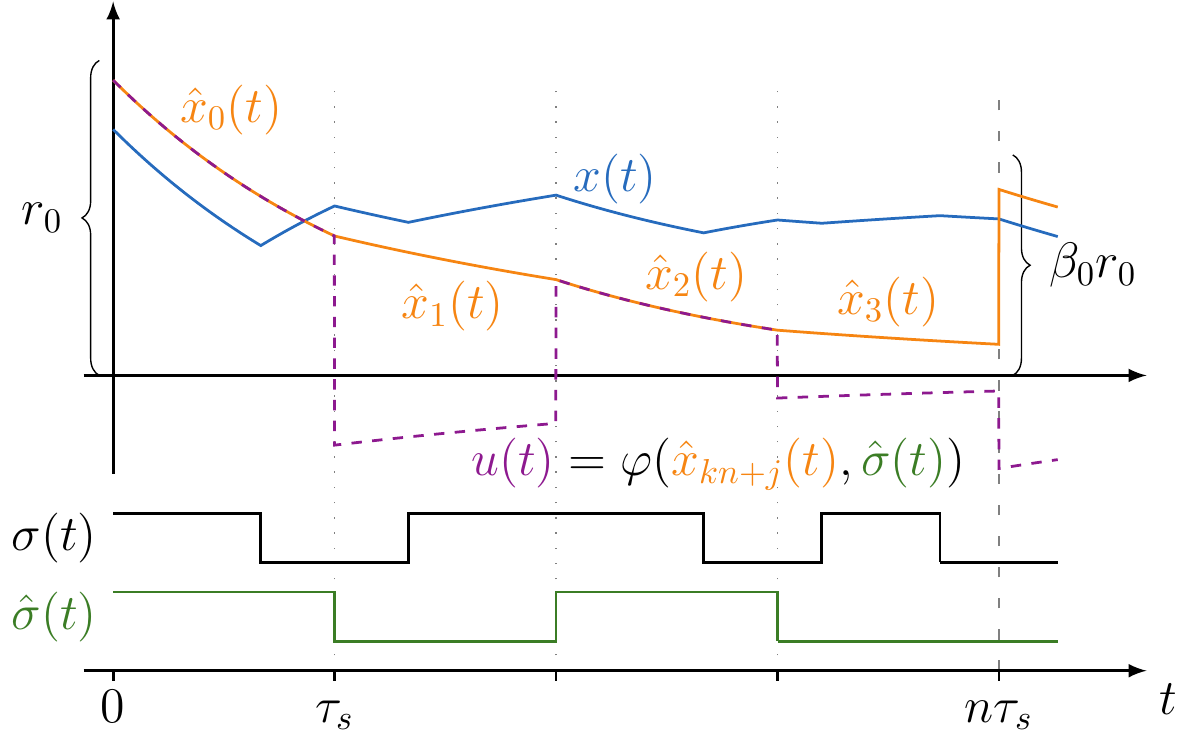}
\caption{
The different quantities involved in the implementation of the controller.
For this example, the coder--controller is applied on the system $\dot{x}(t)=B_{\sigma(t)}u(t)$, where $B_1=-1$ and $B_2=1$, and with $\varphi(x,1)=x$ and $\varphi(x,2)=-x/2$.
The switching signal $\sigma(\cdot)$ is represented in black, and the signal $\sigmah(\cdot)$ obtained from the sampled measurements $\sigma((kn+j)\tau_s)$ (decoded from the symbols $e((kn+j)\tau_s))$ is represented in green.}
\label{fig-iterates}
\vspace{-10pt}
\end{figure}

The implementation deserves the following explanations.
At each time $kn\tau_s$, $k\in\NNb$, the controller receives the symbol $e(kn\tau_s)$ that encodes $\eta_k$, $\sigma(kn\tau_s)$ and $\nmissed_k$.
Based on this, it is able to compute $\beta_k$ and $r_k$, and thus to compute $\xi_{kn}$ which is an approximation of the current state $x(kn\tau_s)$.
Then, on the interval $[kn\tau_s,(kn+1)\tau_s)$, the controller simulates the following auxiliary system:
\begin{equation}\label{eq-auxiliary-system}
\dot{\xh}(t) = A_i\xh(t) + B_i\varphi(\xh(t),i),
\end{equation}
with the boundary condition $\xh(kn\tau_s)=\xi_{kn}$ and with $i=\sigma(kn\tau_s)$.
Using the simulated trajectory, denoted by $\xh_{kn}(\cdot)$, the controller applies on the system the control input defined by $u(t)=\varphi(\xh_{kn}(t),i)$.

The same procedure is repeated for each $j\in\{1,\ldots,n-1\}$: $\xi_{kn+j}$ is defined as $\xh_{kn+j-1}((kn+j)\tau_s)$.
Then, on the interval $[(kn+j)\tau_s,(kn+j+1)\tau_s)$, the controller simulates the auxiliary system \eqref{eq-auxiliary-system} with the boundary condition $\xh((kn+j)\tau_s)=\xi_{kn+j}$ and with $i=\sigma((kn+j)\tau_s)$, decoded from the symbol $e((kn+j)\tau_s)$; and it applies on the system the control input defined by $u(t)=\varphi(\xh_{kn+j}(t),i)$ where $\xh_{kn+j}(\cdot)$ is the simulated trajectory.

%%%%%%%%%%%%%%%%%%%%%%%%%%%%%%%%%%%%%%%%%%%%%%%%%%%%%%%%%%%%%%%%%%%%%%%%%%%%%%%%
\subsection{Proof of the correctness of the coder--controller}\label{ssec-proof-coder-controller-correctness}

\short{See the extended version of this paper \cite{berger2020finite}.}

\extended{

In this subsection, we prove the correctness of the coder--controller.
Concretely, first we prove that for each $k\in\NNb$, there exists $\nmissed_k\in\{0,\ldots,n\}$ such that $\beta_k$ defined by \eqref{eq-beta-defintion} with this $\nmissed_k$ satisfies $\lVert x(kn\tau_s)\rVert\leq\beta_kr_{k-1}$.
Secondly, we prove that the coder--controller stabilizes the system.
For the sake of conciseness, in the following, we will sometimes use the abbreviations $t_k=kn\tau_s$ and $\Theta_k=[t_k,t_{k+1})$.

\begingroup
\renewcommand{\proofname}{Proof of Part~1 (existence of $\nmissed_k$)}
\begin{proof}
For each $j\in\NNb$, let $b^*_j$ be a boolean variable indicating whether or not the mode of the system has changed at least once during the interval $[j\tau_s,(j+1)\tau_s)$: that is, $b^*_j=1$ if $N_\sigma((j+1)\tau_s,j\tau_s)\geq1$, and $b^*_j=0$ otherwise.
For $k\in\NNb$, let $N^*_k = \sum_{j=0}^{n-1} b^*_{kn+j}$.
Clearly, it holds that $N^*_k\leq\min\{n,N_\sigma(t_{k+1},t_k)\}$.
We will show that $\nmissed_k\leq N^*_{k-1}$ (for all $k\geq1$).

Let $\sigmah(\cdot)$ be the switching signal defined by $\sigmah(t)=\sigma(j\tau_s)$ if $t\in[j\tau_s,(j+1)\tau_s)$ with $j\in\NNb$ (i.e., $\sigmah(\cdot)$ is the ``sample and hold'' of $\sigma(\cdot)$ with period $\tau_s$).
Then, for every $k\in\NNb$, it holds that $\int_{\Theta_k}\indfunc_{\sigmah\neq\sigma}(t)\,\diff t\leq N^*_k\tau_s$ where $\indfunc_{\sigmah\neq\sigma}(\cdot)$ is the indicator function of the set $\{t\geq0:\sigmah(t)\neq\sigma(t)\}$.

Fix $k\in\NNb$, and let $\xt_k(\cdot)$ be the solution of the ODE
\[
\left\{\begin{array}{l}
\dot{\xt}_k(t)=A_{\sigmah(t)}\xt_k(t)+B_{\sigmah(t)}\varphi(\xt_k(t),\sigmah(t)), \\
\xt_k(t_k)=\xi_{kn}, \quad t\in\Theta_k\coloneqq[t_k,t_{k+1}).
\end{array}\right.
\]
By Assumption~\ref{ass-feedback-stabilizable}, it holds that
\begin{equation}\label{eq-xt-stable}
\lVert\xt_k(t)\rVert\leq D\lVert\xi_{kn}\rVert e^{\mu_1 N^*_k-\mu_2(t-t_k)}, \quad \forall\,t\in\Theta_k.
\end{equation}
Moreover, by definition of $\xh_{kn+j}(\cdot)$ (see Subsection~\ref{ssec-controller-implementation}), it holds that for every $j\in\{0,\ldots,n-1\}$, $\xh_{kn+j}(t)=\xt_k(t)$ when $t\in[(kn+j)\tau_s,(kn+j+1)\tau_s)$.

Now, assume that $\lVert x(t_k)\rVert\leq r_k$.
We will use the above to show that $\nmissed_{k+1}\leq N^*_k$.
Indeed, by the definition of $\eta_k$ and $\xi_{kn}$, it holds that $\lVert\xi_{kn}\rVert\leq r_k$ and $\lVert x(t_k)-\xi_{kn}\rVert\leq \alpha r_k$.
Thus, by using a classical argument on the sensitivity of solutions of ODEs w.r.t.\ initial conditions and system parameters (see Lemma~\ref{lem-gronwall} in Appendix~\ref{ssec-lem-gronwall}), we get that
\begin{align*}  
& \lVert x(t_{k+1})-\xt_k(t_{k+1})\rVert \leq e^{\nu n\tau_s}\alpha r_k + \\
& \;\qquad\quad \int_{\Theta_k} e^{\nu(t_{k+1}-t)} \lVert \Delta A(t)\xt_k(t)+\Delta B(t) u(t)\rVert \,\diff t
\end{align*}
where $\Delta A(t) = A_{\sigma(t)}-A_{\sigmah(t)}$, $\Delta B(t)=B_{\sigma(t)}-B_{\sigmah(t)}$ and $u(t)=\varphi(\xt_k(t),\sigmah(t))$ ($u(t)$ is the input applied to the system at time $t$).
Combining the above with \eqref{eq-xt-stable} and the definition of $\Delta_1,\Delta_2,L$, we obtain
\begin{align}
& \lVert x(t_{k+1})-\xt_k(t_{k+1})\rVert \leq e^{\nu n\tau_s}\alpha r_k + \nonumber \\
& \qquad\qquad e^{\nu n\tau_s}\int_{\Theta_k} \indfunc_{\sigmah\neq\sigma}(t)(\Delta_1+\Delta_2 L)\lVert\xt_k(t)\rVert \,\diff t \nonumber \\[2pt]
& \quad \leq e^{\nu n\tau_s}\left[\alpha r_k + N^*_k\tau_s (\Delta_1+\Delta_2 L) Dr_k e^{\mu_1 N^*_k}\right]. \label{eq-diff-x-xt}
\end{align}
Moreover, by \eqref{eq-xt-stable}, $\lVert\xt_k(t_{k+1})\rVert\leq Dr_ke^{\mu_1 N^*_k-\mu_2n\tau_s}$.
Hence, combining with \eqref{eq-diff-x-xt}, we get that $\beta_{k+1}$, defined by \eqref{eq-beta-defintion} with $\nmissed_{k+1}=N^*_k$, satisfies $\lVert x(t_{k+1})\rVert\leq\beta_{k+1}r_k$.

To conclude the proof of Part~1: observe that since $k\in\NNb$ is arbitrary and since $\lVert x(t_k)\rVert\leq r_k$ holds true for $k=0$ (by definition of $r_0$), we obtain, with an inductive reasoning, that $\nmissed_k\leq N^*_{k-1}\leq n$ holds true for every $k\in\NNb_{>0}$.
\end{proof}
\endgroup

\begingroup
\renewcommand{\proofname}{Proof of Part~2 (stabilization property)}
\begin{proof}
By definition of $\beta_k$ and $r_k$, it holds that for all $k\in\NNb$, $\lVert x(t_k)\rVert\leq \beta^*_kr_0^{}$ where $\beta^*_k=\prod_{q=1}^k\beta_q^{}$.
We will show that $\beta^*_k\to0$ geometrically as $k\to\infty$.
Indeed, by the definition \eqref{eq-beta-defintion} of $\beta_q$, it holds that for all $k\in\NNb_{>0}$,
\begin{align*}
\textstyle \beta_k^* &\textstyle= e^{\mu_1\sum_{q=1}^k\nmissed_q}\prod_{q=1}^k \big(\psi + e^{-\mu_1\nmissed_q}\alphab +\nmissed_q\epsb\big) \\
&\textstyle\leq e^{\mu_1\sum_{q=1}^k\nmissed_q}\prod_{q=1}^k\big(\psi + \alphab + \nmissed_q\epsb\big).
\end{align*}
From Jensen's inequality,
\[
\textstyle
\prod_{q=1}^k(\psi+\alphab+\nmissed_q\epsb) \leq \big[\frac{1}{k}\sum_{q=1}^k(\psi+\alphab+\nmissed_q\epsb)\big]^k.
\]
Thus, using the fact that $\nmissed_q\leq N^*_{q-1}\leq N_\sigma(t_q,t_{q-1})$ for all $q\in\NNb_{>0}$ (see Part~1), we obtain that for all $k\in\NNb_{>0}$,
\[
\beta_k^* \leq (e^{\mu_1\frac{1}{k}N_\sigma(t_k,0)}\rho_k)^k,\quad\; \rho_k= \psi + \alphab + \tfrac{1}{k}N_\sigma(t_k,0)\epsb.
\]
Now, by Assumption~\ref{ass-ADT} and by condition \eqref{eq-parameters}, we have that $\limsup_{k\to\infty}\rho_k\leq\rhob<e^{-\mu_1n\tau_s/\tau_a}$ where $\rhob$ is the left-hand side term of \eqref{eq-parameters}.
On the other hand, by Assumption~\ref{ass-ADT} again, it holds that $\limsup_{k\to\infty} e^{\mu_1\frac{1}{k}N_\sigma(kn\tau_s,0)}\leq e^{\mu_1n\tau_s/\tau_a}$.
Thus, there is $C\geq0$ such that for every $k\in\NNb$, $\beta_k^*\leq Ce^{-\mu t_k}$ where 
\begin{equation}\label{eq-mu-lambda}
\mu = \mu_1/\tau_a - \log(\rhob)/(n\tau_s) > 0.
\end{equation}
It follows that for all $k\in\NNb$, $\lVert x(t_k)\rVert\leq r_0Ce^{-\mu t_k}$ and the same holds for $\lVert\xi_{kn}\rVert$ (since $\xi_{kn}=r_k\eta_k$ and $\lVert\eta_k\rVert\leq1$).

Now, to obtain an upper bound on $\lVert x(t)\rVert$, observe that from \eqref{eq-xt-stable}, there is $D'\geq0$ such that for every $k\in\NNb$, $\lVert \xt_k(t)\rVert\leq D'r_k$ for all $t\in\Theta_k$.
Since $\lVert\varphi(\xt_k(t),\sigmah(t))\rVert\leq L\lVert\xt_k(t)\rVert$, it follows that $\lVert u(t)\rVert\leq LD'r_k$ for all $t\in\Theta_k$.
Thus, by a similar argument as in the proof of Lemma~\ref{lem-gronwall} and since $\lVert x(t_k)\rVert\leq r_k$, there is $D''\geq0$ such that for every $k\in\NNb$, $\lVert x(t)\rVert\leq D''r_k$ for all $t\in\Theta_k$.
This shows that there is $C'\geq0$ such that
\begin{equation}\label{eq-exponential-convergence}
\lVert x(t)\rVert \leq C'e^{-\mu t} \qquad \forall\,t\geq0.
\end{equation}

Secondly, we claim that the system controlled by the coder--controller is \emph{Lyapunov stable}, meaning that there is a class-$\calK$ function $h(\cdot)$ such that every trajectory of \eqref{eq-SLS-cont} with input defined by \eqref{eq-symbol-coder-controller}--\eqref{eq-input-coder-controller} satisfies $\lVert x(t)\rVert\leq h(\lVert x(0)\rVert)$ for all $t\geq0$.
The proof of this claim is along the same lines as the proof of \cite[Theorem~1]{liberzon2005stabilization}, and thus, omitted here.%
\footnote{In fact, \cite[Theorem~1]{liberzon2005stabilization} shows Lyapunov stability in terms of the ``$\varepsilon$--$\delta$ definition''.
The equivalence of the ``$\varepsilon$--$\delta$ definition'' with the ``class-$\calK$ function definition'' can be found in \cite[Lemma~4.5]{khalil2002nonlinear}, \cite[Lemma~2.5]{clarke1998asymptotic}.}%
\footnote{Property (ii) in Lemma~\ref{lem-m-point-quantizer} and the fact that $\varphi(0,i)=0$ for all $i\in\Sigma$ are crucial here; see \cite[Theorem~1]{liberzon2005stabilization}.}

Finally, we combine the above Lyapunov stability property with the exponential decay property \eqref{eq-exponential-convergence}, to show that the coder--controller stabilizes the system in the sense of Theorem~\ref{thm-coder-controller-stabilizes}.
Therefore, let $\lambda$ be any real such that
\begin{equation}\label{eq-lambda-def}
0 < \lambda < \mu,
\end{equation}
where $\mu$ is as in \eqref{eq-mu-lambda}.
It is readily seen that $\lVert x(t)\rVert=\lVert x(t)\rVert^{1-\lambda/\mu}\lVert x(t)\rVert^{\lambda/\mu}\leq h(\lVert x(0)\rVert)^{1-\lambda/\mu}(C'e^{\mu t})^{\lambda/\mu}$ for all $t\geq0$.
Hence, we get the desired property, taking $g(r)=h(r)^{1-\lambda/\mu} {C'}^{\lambda/\mu}$ which is clearly a class-$\calK$ function.
This concludes the proof of Part~2.
\end{proof}
\endgroup

In this subsection, we have shown that the coder--controller defined in Subsections~\ref{ssec-parameters-definition}--\ref{ssec-controller-implementation} was well defined and that it satisfies the assertions of Theorem~\ref{thm-coder-controller-stabilizes}.
In the next section, we will demonstrate its practical applicability on a numerical example.

}

%%%%%%%%%%%%%%%%%%%%%%%%%%%%%%%%%%%%%%%%%%%%%%%%%%%%%%%%%%%%%%%%%%%%%%%%%%%%%%%%
\section{Numerical experiments}\label{sec-numerical}

Consider the SLS \eqref{eq-SLS-cont} with matrices $A_1=\setMatrix\left[\begin{array}{cc}0.1 & -1.0 \\ 1.5 & 0.1\end{array}\right]$, $A_2=\setMatrix\left[\begin{array}{cc}-0.5 & 2.0 \\ -1.5 & 0.0\end{array}\right]$, $B_1=\setMatrix\left[\begin{array}{c}1 \\ 1\end{array}\right]$, $B_2=\setMatrix\left[\begin{array}{c} 0 \\ 1\end{array}\right]$.
This system is stabilizable, in the sense of Assumption~\ref{ass-feedback-stabilizable}, via the feedback law $u(t) = K_{\sigma(t)}x(t)$ with $K_1=\setMatrix\left[\begin{array}{cc}-0.43 -0.43\end{array}\right]$ and $K_2=\setMatrix\left[\begin{array}{cc}-0.38 -0.52\end{array}\right]$, and with $D=1$, $\mu_1=0$ and $\mu_2=0.15$.

First, we have simulated the system with ADT $\tau_a = 1.0$~s.
We have used the values $\tau_s=0.008$, $\alpha=0.05$ and $n=100$ for the parameters of the coder--controller, which satisfy \eqref{eq-parameters}.
With these values of the parameters, the average data rate of the coder--controller is of $145$~bits/s.
A sample execution of the coder--controller, applied on the system with this ADT, is represented in Figure~\ref{fig-simulate}-(top).
We observe that the state of the system converges to zero, as predicted.

Then, we have simulated the system with a smaller ADT, namely $\tau_a = 0.25$~s.
We have used the values $\tau_s=0.002$, $\alpha=0.05$ and $n=400$ for the parameters of the coder--controller, which satisfy \eqref{eq-parameters}.
The average data rate of the coder--controller is of $523$~bits/s.
A sample execution of the coder--controller, applied on the system with this ADT, is represented in Figure~\ref{fig-simulate}-(bottom).
Again, we observe that the sampled trajectory converges to zero, as predicted.

\begin{figure}
\centering
\includegraphics[width=\linewidth]{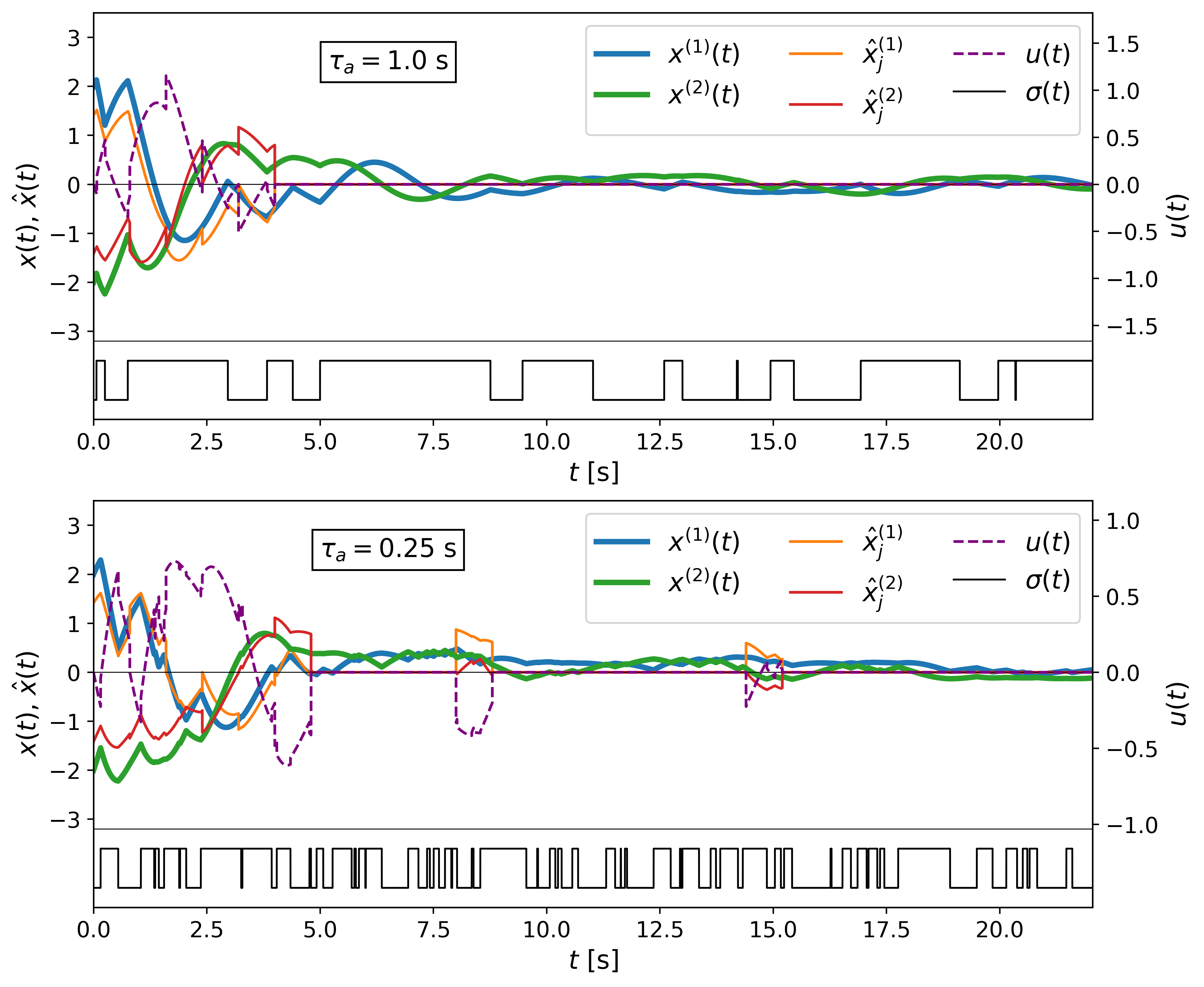}
\caption{Evolution of $x(t)$ and $u(t)$ for a sample execution of the coder--controller applied on the system presented in Section~\ref{sec-numerical}, with different values for the ADT.
The black curve below the plot represents the switching signal.
The orange and red curves represent the trajectories $\xh(t)$ simulated by the controller (see \eqref{eq-auxiliary-system}) to define the input $u(t)$.}
\label{fig-simulate}
\vspace{-10pt}
\end{figure}

%%%%%%%%%%%%%%%%%%%%%%%%%%%%%%%%%%%%%%%%%%%%%%%%%%%%%%%%%%%%%%%%%%%%%%%%%%%%%%%%
\section{Conclusions}

In this paper, we have first shown that continuous-time switched linear systems with arbitrary switching have in general an \emph{infinite stabilization entropy}, meaning that they cannot be stabilized with any finite data rate.
This motivated the introduction of a fairly mild slow-switching assumption on the system.
This assumption requires that the switching signal has an average dwell time bounded away from zero.
Under this assumption, switching linear systems that are stabilizable in the absence of data-rate constraints are stabilizable by a coder--controller with finite data rate.
We have described the implementation of such a coder--controller and demonstrated its applicability on a numerical example.
In future works, the question of potential improvement of the data rate of the coder--controller will be investigated.
As a few examples, one could introduce the use of Lyapunov functions, as in \cite{liberzon2014finite,yang2017feedback,wakaiki2014output}; refine the analysis of the propagation of reachable sets during sampling intervals, by using tools from multilinear algebra, as in \cite{berger2020worstcase}; consider additional assumptions on the system.

%%%%%%%%%%%%%%%%%%%%%%%%%%%%%%%%%%%%%%%%%%%%%%%%%%%%%%%%%%%%%%%%%%%%%%%%%%%%%%%%
\section{Acknowledgments}

The authors would like to thank Daniel Liberzon (University of Illinois in Urbana-Champaign) for insightful discussions on the results presented in the paper.
The first author is also grateful to Prof.~Liberzon for his hospitality while visiting the Dept.~of Electrical and Computer Engineering at UIUC.

\bibliographystyle{plain}
\bibliography{myrefs}

\begin{thebibliography}{10}

\bibitem{berger2020worstcase}
Guillaume~O Berger and Rapha{\"e}l~M Jungers.
\newblock Worst-case topological entropy and minimal data rate for state
  observation of switched linear systems.
\newblock In {\em Proceedings of the 23rd International Conference on Hybrid
  Systems: Computation and Control}, pages 1--11. ACM, 2020.

\bibitem{clarke1998asymptotic}
Francis~H Clarke, Yu~S Ledyaev, and Ronald~J Stern.
\newblock Asymptotic stability and smooth {Lyapunov} functions.
\newblock {\em Journal of Differential Equations}, 149(1):69--114, 1998.

\bibitem{colonius2012minimal}
Fritz Colonius.
\newblock Minimal bit rates and entropy for exponential stabilization.
\newblock {\em SIAM Journal on Control and Optimization}, 50(5):2988--3010,
  2012.

\bibitem{hespanha1999stability}
Jo{\~a}o~P Hespanha and A~Stephen Morse.
\newblock Stability of switched systems with average dwell-time.
\newblock In {\em Proceedings of the 38th IEEE conference on decision and
  control (Cat. No. 99CH36304)}, volume~3, pages 2655--2660. IEEE, 1999.

\bibitem{hespanha2007asurvey}
Jo{\~a}o~P Hespanha, Payam Naghshtabrizi, and Yonggang Xu.
\newblock A survey of recent results in networked control systems.
\newblock {\em Proceedings of the IEEE}, 95(1):138--162, 2007.

\bibitem{jungers2009thejoint}
Rapha{\"e}l~M Jungers.
\newblock {\em The joint spectral radius: theory and applications}, volume 385
  of {\em Lecture Notes in Control and Information Sciences}.
\newblock Springer-Verlag Berlin Heidelberg, 2009.

\bibitem{khalil2002nonlinear}
Hassan~K Khalil.
\newblock {\em Nonlinear systems (3rd ed.)}.
\newblock Prentice-Hall, Upper Saddle River, NJ, 2002.

\bibitem{liberzon2003switching}
Daniel Liberzon.
\newblock {\em Switching in systems and control}.
\newblock Systems \& Control: Foundations \& Applications. Birkh\"auser,
  Boston, MA, 2003.

\bibitem{liberzon2014finite}
Daniel Liberzon.
\newblock Finite data-rate feedback stabilization of switched and hybrid linear
  systems.
\newblock {\em Automatica}, 50(2):409--420, 2014.

\bibitem{liberzon2005stabilization}
Daniel Liberzon and Jo{\~a}o~P Hespanha.
\newblock Stabilization of nonlinear systems with limited information feedback.
\newblock {\em IEEE Transactions on Automatic Control}, 50(6):910--915, 2005.

\bibitem{lin2009stability}
Hai Lin and Panos~J Antsaklis.
\newblock Stability and stabilizability of switched linear systems: a survey of
  recent results.
\newblock {\em IEEE Transactions on Automatic Control}, 54(2):308--322, 2009.

\bibitem{ling2010necessary}
Qiang Ling and Hai Lin.
\newblock Necessary and sufficient bit rate conditions to stabilize quantized
  {Markov} jump linear systems.
\newblock In {\em Proceedings of the 2010 American Control Conference}, pages
  236--240. IEEE, 2010.

\bibitem{matveev2009estimation}
Alexey~S Matveev and Andrey~V Savkin.
\newblock {\em Estimation and control over communication networks}.
\newblock Control Engineering. Birkh\"auser Basel, 2009.

\bibitem{nair2003infimum}
Girish~N Nair, Subhrakanti Dey, and Robin~J Evans.
\newblock Infimum data rates for stabilising {Markov} jump linear systems.
\newblock In {\em 42nd IEEE International Conference on Decision and Control
  (IEEE Cat. No.03CH37475)}, volume~2, pages 1176--1181. IEEE, 2003.

\bibitem{shorten2007stability}
Robert Shorten, Fabian Wirth, Oliver Mason, Kai Wulff, and Christopher King.
\newblock Stability criteria for switched and hybrid systems.
\newblock {\em SIAM Review}, 49(4):545--592, 2007.

\bibitem{sun2011stability}
Zhendong Sun and Shuzhi~Sam Ge.
\newblock {\em Stability theory of switched dynamical systems}.
\newblock Communications and Control Engineering. Springer-Verlag London, 2011.

\bibitem{teschltopics}
Gerald Teschl.
\newblock {\em Topics in real and functional analysis}.
\newblock Graduate studies in mathematics. American Mathematical Society.
\newblock To appear.

\bibitem{van2000anintroduction}
Arjan~J Van Der~Schaft and Johannes~Maria Schumacher.
\newblock {\em An introduction to hybrid dynamical systems}, volume 251 of {\em
  Lecture Notes in Control and Information Sciences}.
\newblock Springer-Verlag London, 2000.

\bibitem{wakaiki2014output}
Masashi Wakaiki and Yutaka Yamamoto.
\newblock Output feedback stabilization of switched linear systems with limited
  information.
\newblock In {\em 53rd IEEE Conference on Decision and Control}, pages
  3892--3897. IEEE, 2014.

\bibitem{xiao2010stabilization}
Nan Xiao, Lihua Xie, and Minyue Fu.
\newblock Stabilization of {Markov} jump linear systems using quantized state
  feedback.
\newblock {\em Automatica}, 46(10):1696--1702, 2010.

\bibitem{yang2017feedback}
Guosong Yang and Daniel Liberzon.
\newblock Feedback stabilization of switched linear systems with unknown
  disturbances under data-rate constraints.
\newblock {\em IEEE Transactions on Automatic Control}, 63(7):2107--2122, 2017.

\bibitem{zhang2009stabilization}
Chun Zhang, Kan Chen, and Geir~E Dullerud.
\newblock Stabilization of {Markovian} jump linear systems with limited
  information --- a convex approach.
\newblock In {\em 2009 American Control Conference}, pages 4013--4019. IEEE,
  2009.

\bibitem{zhongping2013quantized}
Jiang Zhong-Ping and Liu Teng-Fei.
\newblock Quantized nonlinear control --- a survey.
\newblock {\em Acta Automatica Sinica}, 39(11):1820--1830, 2013.

\end{thebibliography}

\extended{

\appendix

%%%%%%%%%%%%%%%%%%%%%%%%%%%%%%%%%%%%%%%%%%%%%%%%%%%%%%%%%%%%%%%%%%%%%%%%%%%%%%%%
\subsection{Proof of Lemma~\ref{lem-m-point-quantizer}}\label{ssec-lem-m-point-quantizer-proof}

Let $S=\{-\llbracket 1/\beta\rrbracket,\ldots,\llbracket 1/\beta\rrbracket\}$ where $\beta=2\alpha/d^{1/2}$.
Clearly, $\lvert S\rvert=2\llbracket 1/\beta\rrbracket+1$.
Then, define $\calQ'=\beta S\times\ldots\times\beta S$ ($d$ times) and $\calQ$ be the orthogonal projection of $\calQ'$ on $B(0,1)$.
Let $Q(\xi)$ be defined as the closest point in $\calQ$ to $\xi$.
Then, it is clear that $Q(\cdot)$ satisfies (i)--(iii).

%%%%%%%%%%%%%%%%%%%%%%%%%%%%%%%%%%%%%%%%%%%%%%%%%%%%%%%%%%%%%%%%%%%%%%%%%%%%%%%%
\subsection{A useful lemma}\label{ssec-lem-gronwall}

The following lemma is inspired from well-known results on the sensitivity of solutions of ODEs to initial conditions and system parameters (see, e.g., \cite[Theorem~3.4]{khalil2002nonlinear}), with a specific adaptation to the case of linear ODEs:

\begin{lemma}\label{lem-gronwall}
Let $x_i:\Re_{\geq0}\to\Re^d$, $i=1,2$, be the solutions of the respective ODEs
\[
\dot{x_i}(t) = A_i(t)x_i(t) + u_i(t), \quad t\geq0, \quad i=1,2,
\]
where $A_i(t)\in\Re^\dd$ and $u_i(t)\in\Re^d$ for all $t\geq0$.
Let $\nu$ be such that $\nu\geq\frac{1}{2}\lambda_{\mathrm{max}}(A_i(t)+A_i(t)^\top)$ for all $t\geq0$ and $i\in\{1,2\}$.
Then, it holds that
\begin{align*}
\lVert x_1(t)-x_2(t)\rVert &\leq e^{\nu t}\,\lVert x_1(0)-x_2(0)\rVert \,+ \\
& \int_0^t e^{\nu(t-s)}\lVert \Delta A(s)x_2(s)+\Delta u(s)\rVert\,\diff s,
\end{align*}
where $\Delta A(s)=A_1(s)-A_2(s)$, $\Delta u(s)=u_1(s)-u_2(s)$.
\end{lemma}

\begin{proof}
Let $w(t)=x_1(t)-x_2(t)$ and $h(t)=\lVert w(t)\rVert$.
Then, $\frac{\diff}{\diff t}(e^{-\nu t}h(t)) = -\nu e^{-\nu t}h(t) + e^{-\nu t} [w(t)^\top\dot{w}(t)]/h(t)$.
It holds that (for simplicity of notation, the dependence on $t$ is assumed implicitly)
\begin{align*}
w^\top \dot{w} & = w^\top[A_1x_1-A_2x_2 + u_1-u_2]  \\
& = w^\top[A_1w + \Delta Ax_2 + \Delta u] \\
& = \frac{1}{2}(w^\top\!A_1^{}w+w^\top\!A_1^\top w) + w^\top(\Delta Ax_2 + \Delta u) \\
& \leq \nu\lVert w\rVert^2 + \lVert \Delta Ax_2 + \Delta u\rVert \lVert w\rVert.
\end{align*}
Injecting this in the expression of $\frac{\diff}{\diff t}(e^{-\nu t}h(t))$, we obtain that for all $t\geq0$, $\frac{\diff}{\diff t}(e^{-\nu t}h(t)) \leq \lVert \Delta A(t)x_2(t) + \Delta u(t)\rVert$.
The conclusion of the proof then follows by integration, and multiplication by $e^{\nu t}$.
\end{proof}

}

\end{document}